\documentclass[11pt,a4paper,twoside]{amsart}
\usepackage[T1]{fontenc}
\usepackage{mathrsfs}
\usepackage{amssymb}
\usepackage{hyperref}

\newcommand{\pB}{\mathcal{B}}
\newcommand{\eH}{\mathscr{H}}
\newcommand{\eK}{\mathscr{K}}
\newcommand{\eM}{\mathscr{M}}
\newcommand{\eN}{\mathscr{N}}
\newcommand{\eR}{\mathscr{R}}
\newcommand{\eX}{\mathscr{X}}
\newcommand{\eY}{\mathscr{Y}}
\newcommand{\eZ}{\mathscr{Z}}
\newcommand{\bC}{\mathbb{C}}
\newcommand{\bM}{\mathbb{M}}
\newcommand{\bdA}{\boldsymbol{A}}
\newcommand{\bdB}{\boldsymbol{B}}
\DeclareMathOperator{\asc}{asc}
\DeclareMathOperator{\dsc}{dsc}
\DeclareMathOperator{\Hol}{Hol}
\DeclareMathOperator{\dist}{dist}
\DeclareMathOperator{\ind}{ind}
\newtheorem{theorem}{Theorem}[section]
\newtheorem{proposition}[theorem]{Proposition}
\newtheorem{lemma}[theorem]{Lemma}
\newtheorem{corollary}[theorem]{Corollary}
\theoremstyle{definition}
\newtheorem{example}[theorem]{Example}
\newtheorem{remark}[theorem]{Remark}
\numberwithin{equation}{section}
\begin{document}
%
\title[Ascent and the angle between the null space and the range]{On the ascent and the angle between the
null space and the range of elementary operators}
\author[J. Bra\v{c}i\v{c}]{Janko Bra\v{c}i\v{c}}
\address{Faculty of Natural Sciences and Engineering, University of Ljubljana, A\v{s}ker\v{c}eva c. 12, SI-1000 Ljubljana, Slovenia}
\email{janko.bracic@ntf.uni-lj.si}
\author[B. Kuzma]{Bojan Kuzma}
\address{University of Primorska, Glagolja\v{s}ka 8, SI-6000 Koper, Slovenia}
\email{bojan.kuzma@famnit.upr.si}
\author[H. Stankovi\'{c}]{Hranislav Stankovi\'{c}$^{*}$}
\address{Faculty of Electronic Engineering, University of Ni\v{s}\\ Aleksandra Medvedeva 4, 18000 Ni\v{s}, Serbia}
\email{hranislav.stankovic@elfak.ni.ac.rs}
\thanks{$^{*}$Corresponding author: Hranislav Stankovi\'{c}.}

\begin{abstract}
Let $\eX$ be a complex Banach space. For an elementary operator $\Delta$ of length at most two on $\pB(\eX)$ we
study the angle $\theta_\Delta$ between the null space and the range of $\Delta$. Its positivity means that these
subspaces have a $k$-gap in the sense of Duggal, and it forces $\asc(\Delta)\leq 1$. For $\mu_{A,B}(X)=AXB$ with
nonzero $A$ and $B$ we prove that $\theta_{\mu_{A,B}}>0$ if and only if $\theta_A>0$ and $\theta_{B^*}>0$, with
explicit bounds, and we compute $\asc(\mu_{A,B})$. If $\sigma(A)\cap\sigma(B)$ is a finite set of poles of the
resolvents of $A$ and $B$, we compute the ascent and the descent of $\delta_{A,B}(X)=AX-XB$ and show that
$\asc(\delta_{A,B})\leq 1$ if and only if $\theta_{\delta_{A,B}}>0$. We also give a mean ergodic estimate of the
angle for $P-Q$ with commuting $P$ and $Q$, and prove that $\delta_{f(A),f(B)}$ and $\delta_{A,B}$ have the same
null spaces, ranges, ascent, descent and angle for every injective holomorphic function $f$.
\end{abstract}
\maketitle

\section{Introduction} \label{Sec01}

Elementary operators include multiplication operators, generalized derivations, and commutators.
Derivations $\delta_A(X)=AX-XA$ and generalized derivations $\delta_{A,B}(X)=AX-XB$ are
closely related to commutator theory. If $[A,[A,T]]=0$, then $[A,T]$ is nilpotent when $\eX$ is finite-dimensional
(Jacobson's lemma \cite[Lemma~2]{Jac}), and it is quasinilpotent in general (the Kleinecke--Shirokov theorem
\cite{Kle, Shi}; see \cite{And, BK, Sta} for generalizations). Since $[A,[A,T]]=0$ means that
$[A,T]\in\eN(\delta_A)\cap\eR(\delta_A)$, these theorems describe the elements of this intersection. The
intersection is trivial if and only if $\asc(\delta_A)\leq 1$. Weber \cite{Web} proved that
$\eR(\delta_{f(A)})\subseteq\eR(\delta_A)$ for every Hilbert space operator $A$ and every function $f$ holomorphic
on a neighborhood of $\sigma(A)$.

Anderson \cite{And} proved that $\|T-(AX-XA)\|\geq\|T\|$ whenever $A$ is a normal operator on a Hilbert space and
$AT=TA$. Thus $\eN(\delta_A)$ is orthogonal to $\eR(\delta_A)$ in the sense of Birkhoff and James \cite{Jam}. This
range--kernel orthogonality was extended in many directions \cite{AF, BB, Du98, Du04, Kec, Kit, Tur}, and its
relation to the finiteness of the ascent was studied in \cite{DDK14, LM}. Duggal \cite{Du04, Du05} introduced a
quantitative version: $\eN(T)$ and $\eR(T)$ have a $k$-gap if $\|y\|\leq k\dist(y,\eR(T))$ for all $y\in\eN(T)$.
This means that the angle $\theta_T$ between $\eN(T)$ and $\eR(T)$ satisfies $\sin\theta_T\geq 1/k$. The link
between orthogonality and ascent goes back to Sinclair \cite{Sin}. He proved that if an eigenvalue $\lambda$ of $T$
lies on the boundary of the numerical range of $T$, then $\eN(T-\lambda)$ is Birkhoff--James orthogonal to
$\eR(T-\lambda)$. Hence $\asc(T-\lambda)\leq 1$ (see also \cite[\S 20]{BD}).

In this paper we study the relation between the conditions $\theta_\Delta>0$ and $\asc(\Delta)\leq 1$ for
elementary operators $\Delta$ of length at most two. The first condition always implies the second one, but not
conversely (Example~\ref{ex01}).

For the multiplication operator $\mu_{A,B}(X)=AXB$ we compute the ascent (Proposition~\ref{prop02}). For nonzero
$A$ and $B$ we prove that $\theta_{\mu_{A,B}}>0$ if and only if $\theta_A>0$ and $\theta_{B^*}>0$, with explicit
two-sided bounds (Theorem~\ref{theo01}).

For a generalized derivation $\delta_{A,B}$ we give a lower bound for the ascent (Proposition~\ref{prop03}). If
$\sigma(A)\cap\sigma(B)$ is a finite set of poles of the resolvents of $A$ and $B$, we compute
$\asc(\delta_{A,B})=\dsc(\delta_{A,B})$ and show that $\asc(\delta_{A,B})\leq 1$ if and only if
$\theta_{\delta_{A,B}}>0$. We also give a formula for $\sin\theta_{\delta_{A,B}}$ (Theorem~\ref{theo04}). The
Drazin invertibility of $\delta_{A,B}$ also follows from results of Boasso and Amouch \cite{BA}, and the
equivalence is an instance of a general property of Drazin invertible operators (Remark~\ref{rem03}). That
simple poles give $\asc(\delta_{A,B})\leq 1$ follows from \cite[Proposition~2.10]{DDK14b}. The converse, the values
of the ascent and the descent, and the formula for the angle seem to be new. For matrices, the value of the ascent agrees with Roth's 
description of the Jordan structure of Kronecker sums \cite[Theorem~2]{Rot} (Corollary~\ref{cor06}). The angle can be arbitrarily small 
even for diagonalizable matrices $A=B$ (Example~\ref{ex02}).

For commuting operators $P$ and $Q$ we give a mean ergodic lower bound for the angle between $\eN(P-Q)$ and
$\eR(P-Q)$ (Theorem~\ref{theo03}). It contains several known range--kernel orthogonality results.

Finally, we show that $f(P)-f(Q)=(P-Q)\Phi_f$ for commuting $P$ and $Q$, a holomorphic function $f$ and an operator
$\Phi_f$ commuting with $P$ and $Q$ (Theorem~\ref{theo02}). Hence, for an injective $f$, the operators
$\delta_{f(A),f(B)}$ and $\delta_{A,B}$ have the same null spaces and ranges of all powers. In particular, they
have the same ascent, descent and angle (Corollary~\ref{cor03}).

Throughout, $\eX$ is a nonzero complex Banach space, $\eX^*$ is its dual space, and
$\langle x,\xi\rangle=\xi(x)$ for $x\in\eX$ and $\xi\in\eX^*$. The Banach algebra of all bounded linear operators
on $\eX$ is $\pB(\eX)$, and $A^*\in\pB(\eX^*)$ denotes the adjoint of $A\in\pB(\eX)$. The null space and the
range of an operator $T$ are denoted by $\eN(T)$ and $\eR(T)$, respectively, and $\sigma(T)$ is the spectrum
of $T$. For $x\in\eX$ and $\xi\in\eX^*$, let $x\otimes\xi$ be the operator $z\mapsto\langle z,\xi\rangle x$.
Then $\|x\otimes\xi\|=\|x\|\|\xi\|$ and $S(x\otimes\xi)T=Sx\otimes T^*\xi$ for all $S,T\in\pB(\eX)$.

\section{Angles between subspaces and ascent} \label{Sec02}

Following \cite{Jam}, a linear submanifold $\eK$ of $\eX$ is Birkhoff--James orthogonal to a linear submanifold
$\eM$, written $\eK\perp\eM$, if $\|x-y\|\geq\|x\|$ for all $x\in\eK$ and $y\in\eM$. Let $\eK\ne\{0\}$ and
$\eM\ne\{0\}$ be linear submanifolds of $\eX$. The angle $\angle(\eK,\eM)\in[0,\frac{\pi}{2}]$ at which $\eK$
meets $\eM$ is defined by
\begin{equation} \label{eq14}
\sin\bigl(\angle(\eK,\eM)\bigr)=\inf\{\|x-y\|;\; x\in\eK,\ \|x\|=1,\ y\in\eM\}.
\end{equation}
If $\eK=\{0\}$ or $\eM=\{0\}$, we set $\sin\bigl(\angle(\eK,\eM)\bigr)=1$. Equivalently,
$\sin\bigl(\angle(\eK,\eM)\bigr)=\inf\{\dist(x,\eM);\; x\in\eK,\ \|x\|=1\}$, and therefore
\begin{equation} \label{eq08}
\|x-y\|\geq\sin\bigl(\angle(\eK,\eM)\bigr)\|x\|\qquad\text{for all}\; x\in\eK,\; y\in\eM.
\end{equation}
Clearly, $\eK\perp\eM$ if and only if $\angle(\eK,\eM)=\frac{\pi}{2}$. If $\angle(\eK,\eM)>0$, then
$\eK\cap\eM=\{0\}$.

\begin{lemma} \label{lem01}
Let $\eK$ and $\eM$ be linear submanifolds of $\eX$ and let $s=\sin\bigl(\angle(\eK,\eM)\bigr)$. Then
$$\sin\bigl(\angle(\eM,\eK)\bigr)\geq\frac{s}{1+s}.$$
In particular, $\angle(\eK,\eM)>0$ if and only if $\angle(\eM,\eK)>0$.
\end{lemma}

\begin{proof}
We may assume that $\eK$ and $\eM$ are nontrivial. Let $y\in\eM$, $\|y\|=1$, and $x\in\eK$. By \eqref{eq08},
$\|y-x\|\geq s\|x\|$, and $\|y-x\|\geq 1-\|x\|$ by the triangle inequality. Hence
$\|y-x\|\geq\min_{t\geq 0}\max\{st,1-t\}=\frac{s}{1+s}$.
\end{proof}

We omit the proof of the following well-known fact.

\begin{lemma}[{\cite[Theorem~1]{Kob}}] \label{lem04}
Let $\eK$ and $\eM$ be closed subspaces of $\eX$ such that $\eK\cap\eM=\{0\}$. Then $\angle(\eK,\eM)>0$
if and only if $\eK+\eM$ is closed.\qed
\end{lemma}

The norm of the corresponding projection determines the angle.

\begin{lemma} \label{lem06}
Let $\eK\ne\{0\}$ and $\eM$ be linear submanifolds of $\eX$ with $\eK\cap\eM=\{0\}$, and let
$E\colon\eK+\eM\to\eK+\eM$ be the projection onto $\eK$ along $\eM$. Then
$$\sin\bigl(\angle(\eK,\eM)\bigr)=\frac{1}{\|E\|},$$
where $\frac{1}{\|E\|}=0$ if $E$ is unbounded.
\end{lemma}

\begin{proof}
Since $\eM$ is a linear submanifold, $\{-y/\|x\|;\; y\in\eM\}=\eM$ for every $0\ne x\in\eK$. Hence
\begin{align*}
\|E\|&=\sup\Bigl\{\frac{\|x\|}{\|x+y\|};\; 0\ne x\in\eK,\ y\in\eM\Bigr\}\\
&=\sup\Bigl\{\frac{1}{\|u-v\|};\; u\in\eK,\ \|u\|=1,\ v\in\eM\Bigr\}=\frac{1}{\sin\bigl(\angle(\eK,\eM)\bigr)}.\qedhere
\end{align*}
\end{proof}

Recall that the ascent $\asc(T)$ of $T\in\pB(\eX)$ is the smallest non-negative integer $k$ such that
$\eN(T^k)=\eN(T^{k+1})$. If there is no such $k$, then $\asc(T)=\infty$ (see \cite[\S 20]{Mul}). For every
non-negative integer $k$, one has $\asc(T)\leq k$ if and only if $\eN(T)\cap\eR(T^k)=\{0\}$. In particular,
$\asc(T)\leq 1$ if and only if $\eN(T)\cap\eR(T)=\{0\}$, and $\asc(T)=0$ if and only if $T$ is injective.
For a nilpotent operator $T$, $\asc(T)$ is its nilpotency index $n(T)=\min\{k;\; T^k=0\}$. We put $n(T)=\infty$
if $T$ is not nilpotent. Similarly, the descent $\dsc(T)$ of $T$ is the smallest non-negative integer $k$ such
that $\eR(T^k)=\eR(T^{k+1})$, and $\dsc(T)=\infty$ if there is no such $k$.

We denote the angle between the null space and the range of $T\in\pB(\eX)$ by $\theta_T$, that is,
$\theta_T=\angle(\eN(T),\eR(T))$. Thus $\theta_T=\frac{\pi}{2}$ if $\eN(T)=\{0\}$, and
$$\sin\theta_T=\inf\{\|e-v\|;\; e\in\eN(T),\ \|e\|=1,\ v\in\eR(T)\}$$
if $\eN(T)\ne\{0\}$. Consequently, $\eN(T)$ and $\eR(T)$ have a $k$-gap in the sense of Duggal \cite{Du05} if and only if
$\sin\theta_T\geq 1/k$, and $\eN(T)\perp\eR(T)$ if and only if $\theta_T=\frac{\pi}{2}$. We will frequently use
the following trivial observation:
\begin{equation} \label{eq01}
\theta_T>0\quad\Longrightarrow\quad\eN(T)\cap\eR(T)=\{0\}\quad\Longleftrightarrow\quad\asc(T)\leq 1.
\end{equation}

\section{Multiplication operators} \label{Sec03}

For $A,B\in\pB(\eX)$, let $L_A\colon X\mapsto AX$ and $R_B\colon X\mapsto XB$ be the left and the right
multiplication operators on $\pB(\eX)$. Then $\|L_A\|=\|A\|$ and $\|R_B\|=\|B\|$. An elementary operator
on $\pB(\eX)$ is an operator of the form $L_{A_1}R_{B_1}+\cdots+L_{A_n}R_{B_n}$. The smallest such $n$ is its
length. We use the notation
$$\mu_{A,B}=L_AR_B,\qquad\delta_{A,B}=L_A-R_B,\qquad\delta_A=\delta_{A,A},\qquad\tau_{A,B}=\mu_{A,B}-I,$$
and $\Delta_{\bdA,\bdB}=\mu_{A_1,B_1}-\mu_{A_2,B_2}$ for pairs $\bdA=(A_1,A_2)$ and $\bdB=(B_1,B_2)$,
where $I$ denotes the identity operator on $\pB(\eX)$. The operators $L_A$ and $R_B$ commute for all
$A,B\in\pB(\eX)$, since $(AX)B=A(XB)$. Consequently, if $A_1A_2=A_2A_1$ and $B_1B_2=B_2B_1$, then the
multiplication operators $\mu_{A_1,B_1}$ and $\mu_{A_2,B_2}$ commute.

\begin{proposition} \label{prop02}
Let $A,B\in\pB(\eX)$.
\begin{itemize}
\item[(i)] Let $k$ be a non-negative integer such that $A^k\ne 0$ and $B^k\ne 0$. Then $\asc(\mu_{A,B})\leq k$
if and only if $\asc(A)\leq k$ and $\asc(B^*)\leq k$.
\item[(ii)] If neither $A$ nor $B$ is nilpotent, then
$$\asc(\mu_{A,B})=\max\{\asc(A),\asc(B^*)\}.$$
\item[(iii)] If $A$ or $B$ is nilpotent, then $\asc(\mu_{A,B})=\min\{n(A),n(B)\}$.
\end{itemize}
\end{proposition}

\begin{proof}
(i) Suppose that $\asc(\mu_{A,B})\leq k$, that is, $A^{k+1}XB^{k+1}=0$ implies $A^kXB^k=0$. Let
$e\in\eN(A^{k+1})$ and choose $\xi\in\eX^*$ with ${B^*}^k\xi\ne 0$, which is possible because $B^k\ne 0$.
For $X=e\otimes\xi$ we have $A^{k+1}XB^{k+1}=A^{k+1}e\otimes{B^*}^{k+1}\xi=0$, and therefore
$0=A^kXB^k=A^ke\otimes{B^*}^k\xi$. Thus $A^ke=0$, which proves $\asc(A)\leq k$. Similarly, if
$\xi\in\eN({B^*}^{k+1})$ and $e\in\eX$ is such that $A^ke\ne 0$, then $X=e\otimes\xi$ satisfies
$A^{k+1}XB^{k+1}=0$, hence $A^ke\otimes{B^*}^k\xi=0$ and ${B^*}^k\xi=0$. Thus $\asc(B^*)\leq k$.

Conversely, assume that $\asc(A)\leq k$ and $\asc(B^*)\leq k$. Let $X\in\pB(\eX)$ be such that
$A^{k+1}XB^{k+1}=0$. Then $XB^{k+1}x\in\eN(A^{k+1})=\eN(A^k)$ for every $x\in\eX$. Hence
$A^kXB^{k+1}=0$. Taking adjoints, we get
${B^*}^{k+1}X^*{A^*}^k=0$. Thus $X^*{A^*}^k\xi\in\eN({B^*}^{k+1})=\eN({B^*}^k)$ for every $\xi\in\eX^*$.
Hence ${B^*}^kX^*{A^*}^k=0$, that is, $A^kXB^k=0$. This proves $\asc(\mu_{A,B})\leq k$. (For $k=0$ we use the
convention $A^0=I$.)

(ii) If neither $A$ nor $B$ is nilpotent, then (i) applies to every $k\geq 0$. Hence $\asc(\mu_{A,B})\leq k$ if and
only if $\max\{\asc(A),\asc(B^*)\}\leq k$, for every $k\geq 0$, which gives the equality (also when one of the
two sides is infinite).

(iii) Let $n=\min\{n(A),n(B)\}<\infty$. Then $\mu_{A,B}^n=L_{A^n}R_{B^n}=0$, and therefore
$\eN(\mu_{A,B}^n)=\eN(\mu_{A,B}^{n+1})=\pB(\eX)$. Since $A^{n-1}\ne 0$ and $B^{n-1}\ne 0$, there are $e\in\eX$
and $\xi\in\eX^*$ with $A^{n-1}e\ne 0$ and ${B^*}^{n-1}\xi\ne 0$, and then
$\mu_{A,B}^{n-1}(e\otimes\xi)=A^{n-1}e\otimes{B^*}^{n-1}\xi\ne 0$. Hence
$\eN(\mu_{A,B}^{n-1})\ne\eN(\mu_{A,B}^n)$ and $\asc(\mu_{A,B})=n$.
\end{proof}

In Proposition~\ref{prop02}(i) and (ii) the ascent of $B^*$, and not of $B$, is relevant. For instance, if $B$ is the unilateral
shift on $\ell^2$, then $\asc(B)=0$ but $\asc(B^*)=\infty$. Hence $\asc(\mu_{I,B})=\infty$.

Here and in the sequel we use the conventions $1/0=\infty$ and $1/\infty=0$.

\begin{theorem} \label{theo01}
Let $A,B\in\pB(\eX)$ be nonzero operators. Then
\begin{equation} \label{eq02}
\Bigl(1+\Bigl(1+\frac{1}{\sin\theta_A}\Bigr)\Bigl(1+\frac{1}{\sin\theta_{B^*}}\Bigr)\Bigr)^{-1}
\leq\sin\theta_{\mu_{A,B}}\leq\min\{\sin\theta_A,\sin\theta_{B^*}\}.
\end{equation}
In particular, $\theta_{\mu_{A,B}}>0$ if and only if $\theta_A>0$ and $\theta_{B^*}>0$.
\end{theorem}

\begin{proof}
\emph{Upper bound.} Assume that $\eN(A)\ne\{0\}$. Otherwise $\sin\theta_A=1$ and there is nothing to prove.
Since $B^*\ne 0$, there is $\eta\in\eX^*$ such that $\xi=B^*\eta$ has norm $1$. Let $e\in\eN(A)$, $\|e\|=1$,
and $f=Au\in\eR(A)$ be arbitrary. Then $T=e\otimes\xi\in\eN(\mu_{A,B})$, $\|T\|=1$, and
$\mu_{A,B}(u\otimes\eta)=Au\otimes B^*\eta=f\otimes\xi$. Hence
$$\sin\theta_{\mu_{A,B}}\leq\|e\otimes\xi-f\otimes\xi\|=\|e-f\|.$$
Taking the infimum over $e$ and $f$ gives $\sin\theta_{\mu_{A,B}}\leq\sin\theta_A$. Similarly, assume that
$\eN(B^*)\ne\{0\}$, and choose $u_0\in\eX$ such that $e_0=Au_0$ has norm $1$. If $\xi'\in\eN(B^*)$, $\|\xi'\|=1$,
and $\zeta=B^*\eta'\in\eR(B^*)$ are arbitrary, then $T=e_0\otimes\xi'\in\eN(\mu_{A,B})$, $\|T\|=1$, and
$\mu_{A,B}(u_0\otimes\eta')=e_0\otimes\zeta$. Hence $\sin\theta_{\mu_{A,B}}\leq\|e_0\otimes(\xi'-\zeta)\|=\|\xi'-\zeta\|$,
and therefore $\sin\theta_{\mu_{A,B}}\leq\sin\theta_{B^*}$.

\emph{Lower bound.} We may assume that $\eN(\mu_{A,B})\ne\{0\}$, $\theta_A>0$ and $\theta_{B^*}>0$. Put
$s_A=\sin\theta_A$, $s_B=\sin\theta_{B^*}$ and $s'=\sin\bigl(\angle(\eR(A),\eN(A))\bigr)$. By Lemma~\ref{lem01},
$\frac{1}{s'}\leq\frac{1+s_A}{s_A}=1+\frac{1}{s_A}$. Let $T\in\eN(\mu_{A,B})$ and $X\in\pB(\eX)$ be arbitrary,
and let $Z=T-AXB$.

Let $\xi\in\eN(A)^\perp=\{\xi\in\eX^*;\; \langle x,\xi\rangle=0\ \text{for all}\ x\in\eN(A)\}$. Since $ATB=0$, we
have $TBx\in\eN(A)$ for every $x\in\eX$, and therefore $\langle x,B^*T^*\xi\rangle=\langle TBx,\xi\rangle=0$, that is,
$T^*\xi\in\eN(B^*)$. Moreover, $(AXB)^*\xi=B^*(X^*A^*\xi)\in\eR(B^*)$. By \eqref{eq08},
$$s_B\|T^*\xi\|\leq\|T^*\xi-(AXB)^*\xi\|=\|Z^*\xi\|\leq\|Z\|\|\xi\|.$$
Hence $|\langle Tx,\xi\rangle|=|\langle x,T^*\xi\rangle|\leq s_B^{-1}\|Z\|\|x\|\|\xi\|$ for all $x\in\eX$ and
$\xi\in\eN(A)^\perp$. By the Hahn--Banach theorem,
$\dist(y,\eN(A))=\sup\{|\langle y,\xi\rangle|;\; \xi\in\eN(A)^\perp,\ \|\xi\|\leq 1\}$ for every $y\in\eX$, and thus
$$\dist(Tx,\eN(A))\leq s_B^{-1}\|Z\|\|x\|\qquad(x\in\eX).$$
Since $AXBx=Tx-Zx$, it follows that $\dist(AXBx,\eN(A))\leq(1+s_B^{-1})\|Z\|\|x\|$. On the other hand,
$AXBx\in\eR(A)$, and therefore \eqref{eq08} gives $s'\|AXBx\|\leq\dist(AXBx,\eN(A))$. Consequently,
$$\|AXB\|\leq\frac{1}{s'}\Bigl(1+\frac{1}{s_B}\Bigr)\|Z\|\leq\Bigl(1+\frac{1}{s_A}\Bigr)\Bigl(1+\frac{1}{s_B}\Bigr)\|Z\|,$$
and finally
$$\|T\|\leq\|AXB\|+\|Z\|\leq\Bigl(1+\Bigl(1+\frac{1}{s_A}\Bigr)\Bigl(1+\frac{1}{s_B}\Bigr)\Bigr)\|T-AXB\|.$$
Since $T\in\eN(\mu_{A,B})$ and $X\in\pB(\eX)$ were arbitrary, the lower bound in \eqref{eq02} follows.
\end{proof}

The following example shows that the converse of \eqref{eq01} fails already for multiplication operators.

\begin{example} \label{ex01}
Let $\eH$ be a separable infinite-dimensional Hilbert space and let $C\in\pB(\eH)$ be an injective compact operator
with dense range, for instance a diagonal operator with eigenvalues $\frac{1}{n}$. In $\eH\oplus\eH$, let
$\eK=\eH\oplus\{0\}$ and let $\eM=\{(x,Cx);\; x\in\eH\}$ be the graph of $C$. Then $\eK$ and $\eM$ are closed
subspaces, $\eK\cap\eM=\{0\}$, and $\eK+\eM=\eH\oplus\eR(C)$ is not closed. Hence $\angle(\eK,\eM)=0$ by
Lemma~\ref{lem04}. Since $\eK^\perp=\{0\}\oplus\eH$ and $\eM$ are both infinite-dimensional and separable, there is a
partial isometry $A$ on $\eH\oplus\eH$ with initial space $\eK^\perp$ and final space $\eM$, that is,
$\eN(A)=\eK$ and $\eR(A)=\eM$. Then $\eN(A)\cap\eR(A)=\{0\}$ and $\eN(A)\ne\{0\}$. Hence $\asc(A)=1$, while
$\theta_A=0$.

Let $B\ne 0$ be any operator with $\asc(B^*)\leq 1$, for instance $B=I$. By Proposition~\ref{prop02}(i) with $k=1$,
$\asc(\mu_{A,B})\leq 1$, and $\mu_{A,B}$ is not injective since $A$ is not injective. Thus
$\asc(\mu_{A,B})=1$, whereas $\theta_{\mu_{A,B}}=0$ by Theorem~\ref{theo01}.
\end{example}

\section{Generalized derivations} \label{Sec04}

In contrast to Example~\ref{ex01}, we show in this section that for a large class of generalized derivations the
conditions $\asc(\delta_{A,B})\leq 1$ and $\theta_{\delta_{A,B}}>0$ are equivalent. We start with a lower
bound for the ascent.

\begin{proposition} \label{prop03}
Let $A,B\in\pB(\eX)$ and $\lambda\in\bC$. Let $m,n\geq 1$ be integers such that
$\eN\bigl((A-\lambda)^{m-1}\bigr)\ne\eN\bigl((A-\lambda)^m\bigr)$ and
$\eN\bigl((B^*-\lambda)^{n-1}\bigr)\ne\eN\bigl((B^*-\lambda)^n\bigr)$. Then $\asc(\delta_{A,B})\geq m+n-1$.
In particular, if $\lambda$ is an eigenvalue of $A$ and of $B^*$, then
$$\asc(\delta_{A,B})\geq\asc(A-\lambda)+\asc(B^*-\lambda)-1.$$
Consequently, if $\asc(\delta_{A,B})\leq 1$, then $\asc(A-\lambda)\leq 1$ and $\asc(B^*-\lambda)\leq 1$ for every
common eigenvalue $\lambda$ of $A$ and $B^*$.
\end{proposition}

\begin{proof}
Put $N=A-\lambda$ and $M=B-\lambda$. Then $\delta_{A,B}=L_N-R_M$, and since $L_N$ and $R_M$ commute,
\begin{equation} \label{eq03}
\delta_{A,B}^k(X)=\sum_{j=0}^{k}\binom{k}{j}(-1)^{k-j}N^jXM^{k-j}\qquad(X\in\pB(\eX),\ k\geq 0).
\end{equation}
Choose $x\in\eN(N^m)\setminus\eN(N^{m-1})$ and $\xi\in\eN({M^*}^n)\setminus\eN({M^*}^{n-1})$, and let
$X=x\otimes\xi$. Then $N^jXM^{k-j}=N^jx\otimes{M^*}^{k-j}\xi$. If $k=m+n-1$, then for every $j$ either $j\geq m$ or
$k-j\geq n$. Thus all terms in \eqref{eq03} vanish and $\delta_{A,B}^{m+n-1}(X)=0$. If $k=m+n-2$, then the only
term that might not vanish is the one with $j=m-1$, and it equals
$\binom{m+n-2}{m-1}(-1)^{n-1}N^{m-1}x\otimes{M^*}^{n-1}\xi\ne 0$. Hence
$X\in\eN(\delta_{A,B}^{m+n-1})\setminus\eN(\delta_{A,B}^{m+n-2})$, which gives $\asc(\delta_{A,B})\geq m+n-1$.

If $\lambda$ is an eigenvalue of $A$ and of $B^*$, then $\asc(A-\lambda)\geq 1$, $\asc(B^*-\lambda)\geq 1$, and
the null spaces of the powers of $A-\lambda$ (respectively $B^*-\lambda$) strictly increase up to the ascent.
Hence the first part applies to all integers $1\leq m\leq\asc(A-\lambda)$ and $1\leq n\leq\asc(B^*-\lambda)$,
which gives the second assertion (also when one of the ascents is infinite). The last assertion follows
immediately.
\end{proof}

Let $A\in\pB(\eX)$ and let $\lambda$ be an isolated point of $\sigma(A)$. We denote by $P_\lambda(A)$ the Riesz
idempotent of $A$ at $\lambda$ (see \cite[Ch.~VII, \S 4]{Con} or \cite[\S V.9]{TL}). Recall that $\lambda$ is a
pole of the resolvent $(z-A)^{-1}$ of order $p\geq 1$ if and only if $\asc(A-\lambda)=\dsc(A-\lambda)=p$, and
then $(A-\lambda)^pP_\lambda(A)=0\ne(A-\lambda)^{p-1}P_\lambda(A)$ and
$\eR(P_\lambda(A))=\eN\bigl((A-\lambda)^p\bigr)$ (see \cite[\S V.10]{TL}). Since
$(z-A^*)^{-1}=\bigl((z-A)^{-1}\bigr)^*$ for $z\notin\sigma(A)=\sigma(A^*)$, the point $\lambda$ is a pole of
order $p$ of the resolvent of $A$ if and only if it is a pole of order $p$ of the resolvent of $A^*$.

We need Rosenblum's theorem \cite{Ros} for operators acting between two Banach spaces.
Its standard proof, based on the formula $S=\frac{1}{2\pi i}\int_\Gamma(z-A)^{-1}C(z-B)^{-1}\,dz$ for the solution
of $AS-SB=C$, where $\Gamma$ is a cycle surrounding $\sigma(A)$ but not $\sigma(B)$, works without
changes in this setting. We state it without proof.

\begin{lemma} \label{lem07}
Let $\eY$ and $\eZ$ be Banach spaces and let $A\in\pB(\eY)$ and $B\in\pB(\eZ)$ be such that
$\sigma(A)\cap\sigma(B)=\emptyset$. Then the operator $S\mapsto AS-SB$ is a bijection of $\pB(\eZ,\eY)$.\qed
\end{lemma}

Now we can prove the main result of this section.

\begin{theorem} \label{theo04}
Let $A,B\in\pB(\eX)$ be such that $\sigma(A)\cap\sigma(B)=\{\lambda_1,\ldots,\lambda_r\}$ is a finite (possibly
empty) set and, for each $i=1,\ldots,r$, $\lambda_i$ is a pole of the resolvent of $A$ of order $p_i$ and a pole of
the resolvent of $B$ of order $q_i$.
\begin{itemize}
\item[(i)] $\asc(\delta_{A,B})=\dsc(\delta_{A,B})=\max\{p_i+q_i-1;\; 1\leq i\leq r\}$, where the maximum of the
empty set is $0$.
\item[(ii)] The following assertions are equivalent:
\begin{itemize}
\item[(a)] $\theta_{\delta_{A,B}}>0$;
\item[(b)] $\asc(\delta_{A,B})\leq 1$;
\item[(c)] $p_i=q_i=1$ for $i=1,\ldots,r$;
\item[(d)] $\pB(\eX)=\eN(\delta_{A,B})\oplus\eR(\delta_{A,B})$, a topological direct sum of closed subspaces.
\end{itemize}
\item[(iii)] If the assertions in (ii) hold and $r\geq 1$, then
$$\sin\theta_{\delta_{A,B}}=\frac{1}{\|E\|}\geq\Bigl(\sum_{i=1}^r\|P_{\lambda_i}(A)\|\|P_{\lambda_i}(B)\|\Bigr)^{-1},$$
where $E(X)=\sum_{i=1}^rP_{\lambda_i}(A)XP_{\lambda_i}(B)$ for $X\in\pB(\eX)$.
\end{itemize}
\end{theorem}

\begin{proof}
Let $P_i=P_{\lambda_i}(A)$ and $Q_i=P_{\lambda_i}(B)$ for $i=1,\ldots,r$, and let $P_0=I-\sum_{i=1}^rP_i$ and
$Q_0=I-\sum_{i=1}^rQ_i$. Then $P_0,\ldots,P_r$ are mutually annihilating idempotents with sum $I$ which commute
with $A$, and $\eX_i=\eR(P_i)$ is invariant for $A$. Moreover, $\sigma(A|_{\eX_i})=\{\lambda_i\}$ for $i\geq 1$
and $\sigma(A|_{\eX_0})=\sigma(A)\setminus\{\lambda_1,\ldots,\lambda_r\}$ (see \cite[Ch.~VII, \S 4]{Con}). The same
holds for $Q_0,\ldots,Q_r$, $B$ and $\eZ_j=\eR(Q_j)$.

For $i,j\in\{0,\ldots,r\}$, let $E_{ij}(X)=P_iXQ_j$ and $\eY_{ij}=\eR(E_{ij})$. The operators $E_{ij}$ are mutually
annihilating idempotents on $\pB(\eX)$ with sum $I$. Hence $\pB(\eX)$ is the topological direct sum of the closed
subspaces $\eY_{ij}$. Since $P_i$ commutes with $A$ and $Q_j$ commutes with $B$, we have
$\delta_{A,B}E_{ij}=E_{ij}\delta_{A,B}$. Thus each $\eY_{ij}$ is invariant for $\delta_{A,B}$. Let $\delta_{ij}$ denote
the restriction of $\delta_{A,B}$ to $\eY_{ij}$, and let $D=\{(i,i);\; 1\leq i\leq r\}$.

\emph{Step 1: $\delta_{ij}$ is bijective for $(i,j)\notin D$.} One checks directly that
$\sigma(A|_{\eX_i})\cap\sigma(B|_{\eZ_j})=\emptyset$ for $(i,j)\notin D$. Indeed, for $i,j\geq 1$, $i\ne j$, the two
spectra are $\{\lambda_i\}$ and $\{\lambda_j\}$. If exactly one of $i,j$ equals $0$, then one spectrum is a
point $\lambda_k$ and the other one does not contain $\lambda_k$. If $i=j=0$, then a common point would
belong to $\sigma(A)\cap\sigma(B)\setminus\{\lambda_1,\ldots,\lambda_r\}=\emptyset$. The map
$J\colon\eY_{ij}\to\pB(\eZ_j,\eX_i)$, $J(X)=X|_{\eZ_j}$, is a bounded linear bijection with the inverse
$S\mapsto SQ_j$ (where $S$ is regarded as an operator into $\eX$), and
$J\delta_{ij}J^{-1}(S)=A|_{\eX_i}S-SB|_{\eZ_j}$. By Lemma~\ref{lem07}, $\delta_{ij}$ is bijective.

\emph{Step 2: $\delta_{ii}^{p_i+q_i-1}=0$ for $1\leq i\leq r$.} Let $N_i=(A-\lambda_i)P_i$ and
$M_i=(B-\lambda_i)Q_i$. Then $N_i^{p_i}=0$ and $M_i^{q_i}=0$, and for $X\in\eY_{ii}$ we have
$\delta_{A,B}(X)=(A-\lambda_i)P_iX-XQ_i(B-\lambda_i)=N_iX-XM_i$. Since $L_{N_i}$ and $R_{M_i}$ commute,
$\delta_{ii}^k(X)=\sum_{l=0}^k\binom{k}{l}(-1)^{k-l}N_i^lXM_i^{k-l}$, and every term vanishes when
$k\geq p_i+q_i-1$.

(i) Let $m=\max\{p_i+q_i-1;\; 1\leq i\leq r\}$ ($m=0$ if $r=0$) and suppose that $\delta_{A,B}^{m+1}(X)=0$. Then
$\delta_{ij}^{m+1}(E_{ij}X)=E_{ij}\delta_{A,B}^{m+1}(X)=0$ for all $i,j$. By Step 1, $E_{ij}X=0$ for
$(i,j)\notin D$, and by Step 2, $\delta_{A,B}^m(E_{ii}X)=0$ for $1\leq i\leq r$. Hence
$\delta_{A,B}^m(X)=\sum_{i,j}\delta_{A,B}^m(E_{ij}X)=0$, which proves $\asc(\delta_{A,B})\leq m$. On the other
hand, for each $i\geq 1$, $\lambda_i$ is an eigenvalue of $A$ and of $B^*$ with $\asc(A-\lambda_i)=p_i$ and
$\asc(B^*-\lambda_i)=q_i$, because $\lambda_i$ is a pole of order $q_i$ of the resolvent of $B^*$. By
Proposition~\ref{prop03}, $\asc(\delta_{A,B})\geq p_i+q_i-1$. Thus $\asc(\delta_{A,B})=m$.

It remains to compute the descent. Let $\eY'=\bigoplus_{(i,j)\notin D}\eY_{ij}$ and $\eY''=\bigoplus_{i=1}^r\eY_{ii}$,
and let $\delta'$ and $\delta''$ be the restrictions of $\delta_{A,B}$ to $\eY'$ and $\eY''$, respectively. Then
$\pB(\eX)=\eY'\oplus\eY''$, the operator $\delta'$ is bijective by Step 1, and ${\delta''}^m=0$ by Step 2. Hence
$$\eR(\delta_{A,B}^k)=\eR({\delta'}^k)\oplus\eR({\delta''}^k)=\eY'\oplus\eR({\delta''}^k)\qquad(k\geq 0).$$
In particular, $\eR(\delta_{A,B}^k)=\eY'$ for all $k\geq m$. If $m=0$, this gives $\dsc(\delta_{A,B})=0$. Let $m\geq 1$.
Since $\asc(\delta_{A,B})=m$, there is $X\in\eN(\delta_{A,B}^m)\setminus\eN(\delta_{A,B}^{m-1})$. Then
$\delta_{ij}^m(E_{ij}X)=E_{ij}\delta_{A,B}^m(X)=0$ for all $i,j$. By Step 1, $E_{ij}X=0$ for $(i,j)\notin D$, that is,
$X\in\eY''$. Consequently, ${\delta''}^{m-1}X=\delta_{A,B}^{m-1}X\ne 0$, and therefore
$$\eR(\delta_{A,B}^{m-1})=\eY'\oplus\eR({\delta''}^{m-1})\ne\eY'=\eR(\delta_{A,B}^m).$$
This proves $\dsc(\delta_{A,B})=m$.

(ii) The implication (a)$\Rightarrow$(b) is \eqref{eq01}, and (b)$\Rightarrow$(c) follows from (i), since
$p_i,q_i\geq 1$. Assume (c). Then $N_i=0$ and $M_i=0$, and hence $\delta_{ii}=0$ for $1\leq i\leq r$. Together with Step 1
this gives
$$\eN(\delta_{A,B})=\bigoplus_{(i,j)\in D}\eY_{ij}\qquad\text{and}\qquad\eR(\delta_{A,B})=\bigoplus_{(i,j)\notin D}\eY_{ij},$$
which proves (d). Moreover, $E=\sum_{i=1}^rE_{ii}$ is the bounded projection onto $\eN(\delta_{A,B})$ along
$\eR(\delta_{A,B})$. Finally, if (d) holds, then the projection onto $\eN(\delta_{A,B})$ along
$\eR(\delta_{A,B})$ is bounded by the closed graph theorem, and $\theta_{\delta_{A,B}}>0$ by Lemma~\ref{lem06}
(or by Lemma~\ref{lem04}).

(iii) By the proof of (ii) and Lemma~\ref{lem06}, $\sin\theta_{\delta_{A,B}}=1/\|E\|$, and clearly
$\|E\|\leq\sum_{i=1}^r\|P_i\|\|Q_i\|$.
\end{proof}

\begin{remark} \label{rem03}
By Theorem~\ref{theo04}(i), $\delta_{A,B}$ is Drazin invertible, and if $r\geq 1$, then $0$ is a pole of order
$\max\{p_i+q_i-1;\; 1\leq i\leq r\}$ of the resolvent of $\delta_{A,B}$. The Drazin invertibility of $\delta_{A,B}$
can also be deduced from the description of the Drazin spectrum of a generalized derivation given by Boasso and
Amouch \cite[Theorems~3.5 and 3.7]{BA}. The decomposition of $\pB(\eX)$ used in the proof
above appears there and in \cite[Proposition~3.1]{DDK14} as well. If $A$ and $B$ are simply polaroid, that is, every isolated point of their
spectra is a pole of order one, then $\asc(\delta_{A,B})\leq 1$ and $\delta_{A,B}$ has closed range whenever
$0\in\operatorname{iso}\sigma(\delta_{A,B})$ \cite[Proposition~2.10]{DDK14b}. Related equivalences, under the assumption $\asc(\delta_{A,B})\leq 1$,
are proved in \cite[Theorem~3.2 and Remark~3.3]{DDK14} and \cite[Theorem~4.4]{LM}.

For every Drazin invertible operator $T$ on a Banach space $\eY$, the conditions $\asc(T)\leq 1$, $\theta_T>0$
and $\eY=\eN(T)\oplus\eR(T)$ (a topological direct sum) are equivalent. This follows from \eqref{eq01},
Lemma~\ref{lem06} and the decomposition $\eY=\eN(T^k)\oplus\eR(T^k)$ for $k=\asc(T)=\dsc(T)$ (see
\cite[\S 20]{Mul}). Thus the equivalences in Theorem~\ref{theo04}(ii) express this general property in terms of the orders of the
poles of the resolvents of $A$ and $B$.\qed
\end{remark}

Recall that $A\in\pB(\eX)$ is algebraic if $p(A)=0$ for a nonzero polynomial $p$. This holds if and only if
$\sigma(A)$ is finite and each point of $\sigma(A)$ is a pole of the resolvent of $A$. In this case the minimal
polynomial of $A$ is $m_A(z)=\prod_{\lambda\in\sigma(A)}(z-\lambda)^{p_\lambda}$, where $p_\lambda$ is the order
of the pole $\lambda$. Every operator on a finite-dimensional space is algebraic. Recall that $\delta_{A,B}$ is
algebraic if and only if $A$ and $B$ are algebraic \cite[Proposition~2.6]{DDK14b}. Theorem~\ref{theo04} gives the
following.

\begin{corollary} \label{cor06}
Let $A,B\in\pB(\eX)$ be algebraic operators with minimal polynomials $m_A$ and $m_B$, and for $\lambda\in\bC$
let $p_\lambda$ and $q_\lambda$ be the multiplicities of $\lambda$ as a root of $m_A$ and $m_B$, respectively.
Then
$$\asc(\delta_{A,B})=\dsc(\delta_{A,B})=\max\{p_\lambda+q_\lambda-1;\; \lambda\in\sigma(A)\cap\sigma(B)\}$$
(where the maximum of the empty set is $0$). Moreover, the following assertions are equivalent:
\begin{itemize}
\item[(a)] $\asc(\delta_{A,B})\leq 1$;
\item[(b)] $\theta_{\delta_{A,B}}>0$;
\item[(c)] every common root of $m_A$ and $m_B$ is a simple root of both.
\end{itemize}
In particular,
$\asc(\delta_A)=\dsc(\delta_A)=2\max\{p_\lambda;\; \lambda\in\sigma(A)\}-1$ is odd, and $\asc(\delta_A)\leq 1$ if and
only if $m_A$ has only simple roots.\qed
\end{corollary}

Let $\bM_n$ denote the space of all $n\times n$ complex matrices. For $A,B\in\bM_n$, the value of $\asc(\delta_{A,B})$ in Corollary~\ref{cor06}
can also be read off from the Jordan structure of $\delta_{A,B}$, which is similar to the Kronecker sum
$A\otimes I_n-I_n\otimes B^{T}$. By a classical theorem of Roth \cite[Theorem~2]{Rot}, if $J_p(\lambda)$ denotes the Jordan
block of size $p$ with eigenvalue $\lambda$, then the nilpotent matrix $J_p(\lambda)\otimes I_q-I_p\otimes J_q(\lambda)$
has Jordan blocks of sizes $p+q-1,\,p+q-3,\ldots,|p-q|+1$. The largest of them is of size $p+q-1$, in accordance
with Corollary~\ref{cor06}. For a matrix $A\in\bM_n$, Corollary~\ref{cor06} says that $\asc(\delta_A)\leq 1$ if and only
if $A$ is diagonalizable. Thus, for a diagonalizable matrix $A$, the condition $[A,[A,T]]=0$ already implies
$[A,T]=0$, which is a sharpening of Jacobson's lemma in this special case.

\begin{example} \label{ex02}
Let $\bC^2$ be equipped with the Euclidean norm. Then $\pB(\bC^2)=\bM_2$ carries the operator norm.
For $t>0$ let $P_t=\left(\begin{smallmatrix}1&t\\0&0\end{smallmatrix}\right)$.
Then $P_t^2=P_t$ and $\sigma(P_t)=\{0,1\}$, both points being poles of order one, with the Riesz idempotents
$P_1(P_t)=P_t$ and $P_0(P_t)=I-P_t$. Let $A=B=P_t$. By Theorem~\ref{theo04}, $\asc(\delta_{P_t})=1$,
$\theta_{\delta_{P_t}}>0$, and $\sin\theta_{\delta_{P_t}}=1/\|E_t\|$, where
$$E_t(X)=P_tXP_t+(I-P_t)X(I-P_t)\qquad(X\in\bM_2).$$
For $X=\left(\begin{smallmatrix}0&0\\1&0\end{smallmatrix}\right)$ we have $\|X\|=1$ and a direct computation gives
$E_t(X)=\left(\begin{smallmatrix}t&2t^2\\0&-t\end{smallmatrix}\right)$,
and therefore $\|E_t\|\geq\|E_t(X)\|\geq 2t^2$. On the other hand, $\|P_t\|=\|I-P_t\|=\sqrt{1+t^2}$. Thus the lower
bound in Theorem~\ref{theo04}(iii) equals $\bigl(2(1+t^2)\bigr)^{-1}$. Hence
$$\frac{1}{2(1+t^2)}\leq\sin\theta_{\delta_{P_t}}\leq\frac{1}{2t^2}.$$
Thus $\theta_{\delta_{P_t}}\to 0$ as $t\to\infty$, although $\asc(\delta_{P_t})=1$ and $\theta_{\delta_{P_t}}>0$
for every $t>0$. Moreover, the lower bound in Theorem~\ref{theo04}(iii) is attained up to the factor $(1+t^2)/t^2$.\qed
\end{example}

\section{A mean ergodic estimate} \label{Sec05}

In this section $P$ and $Q$ are commuting operators on a Banach space $\eY$. The main example is $\eY=\pB(\eX)$,
$P=\mu_{A_1,B_1}$ and $Q=\mu_{A_2,B_2}$, where $A_1A_2=A_2A_1$ and $B_1B_2=B_2B_1$. Then $P-Q=\Delta_{\bdA,\bdB}$.
For $n\geq 1$ let
$$S_n=\sum_{i=0}^{n-1}P^{n-1-i}Q^i.$$
Since $P$ and $Q$ commute, $S_n(P-Q)=P^n-Q^n$. If $T\in\eN(P-Q)$, then $Q^iT=P^iT$ for all $i\geq 0$, and hence
$S_nT=nP^{n-1}T$. The following argument is the one used in the proof of the mean ergodic theorem (see
\cite[Chapter~2]{Kre}).

\begin{theorem} \label{theo03}
Let $P,Q\in\pB(\eY)$ be commuting operators such that $\eN(P-Q)\ne\{0\}$. Suppose that there are constants
$K,M>0$ and a strictly increasing sequence $(n_k)_{k=1}^\infty$ of positive integers such that
\begin{itemize}
\item[(1)] $\|S_{n_k}\|\leq Kn_k$ for all $k$;
\item[(2)] $\lim_{k\to\infty}\frac{1}{n_k}\bigl(\|P^{n_k}\|+\|Q^{n_k}\|\bigr)=0$;
\item[(3)] $\|P^{n_k-1}T\|\geq M\|T\|$ for all $T\in\eN(P-Q)$ and all $k$.
\end{itemize}
Then $\sin\bigl(\angle(\eN(P-Q),\eR(P-Q))\bigr)\geq\frac{M}{K}$. In particular, $\asc(P-Q)=1$.
\end{theorem}

\begin{proof}
Let $T\in\eN(P-Q)$ and $X\in\eY$. Then
$$n_kP^{n_k-1}T=S_{n_k}T=S_{n_k}\bigl(T-(P-Q)X\bigr)+\bigl(P^{n_k}-Q^{n_k}\bigr)X,$$
and therefore, by (1) and (3),
$$n_kM\|T\|\leq Kn_k\|T-(P-Q)X\|+\bigl(\|P^{n_k}\|+\|Q^{n_k}\|\bigr)\|X\|.$$
Dividing by $n_k$ and letting $k\to\infty$, we obtain $M\|T\|\leq K\|T-(P-Q)X\|$ by (2). Since
$\eN(P-Q)\ne\{0\}$, the angle is positive and $\asc(P-Q)=1$ by \eqref{eq01}.
\end{proof}

The assumption $\eN(P-Q)\ne\{0\}$ is needed, since otherwise condition (3) holds for every $M>0$.

Recall that $T\in\pB(\eY)$ is power-bounded if $\sup_{n\geq 0}\|T^n\|<\infty$, and that $T$ is power-bounded below
if there is $M>0$ such that $\|T^ny\|\geq M\|y\|$ for all $y\in\eY$ and $n\geq 0$. If $A$ and $B$ are
power-bounded, then so is $\mu_{A,B}$.

\begin{corollary} \label{cor04}
Let $P,Q\in\pB(\eY)$ be commuting power-bounded operators such that $\|P^n\|\leq K_1$ and
$\|Q^n\|\leq K_2$ for all $n\geq 0$, and let $\eN(P-Q)\ne\{0\}$. If the restriction of $P$ to the invariant
subspace $\eN(P-Q)$ is power-bounded below with a constant $M>0$, then
$$\sin\bigl(\angle(\eN(P-Q),\eR(P-Q))\bigr)\geq\frac{M}{K_1K_2}.$$
In particular, this holds for $\Delta_{\bdA,\bdB}=\mu_{A_1,B_1}-\mu_{A_2,B_2}$, where $A_1A_2=A_2A_1$,
$B_1B_2=B_2B_1$, $\mu_{A_1,B_1}$ and $\mu_{A_2,B_2}$ are power-bounded, $\eN(\Delta_{\bdA,\bdB})\ne\{0\}$,
and the restriction of $\mu_{A_1,B_1}$ to $\eN(\Delta_{\bdA,\bdB})$ is power-bounded below.
\end{corollary}

\begin{proof}
We have $\|S_n\|\leq nK_1K_2$ and $\|P^n\|+\|Q^n\|\leq K_1+K_2$. Hence Theorem~\ref{theo03} applies with $n_k=k$ and
$K=K_1K_2$.
\end{proof}

If $Q=I$, then condition (3) of Theorem~\ref{theo03} holds with $M=1$, since $P^nT=T$ for $T\in\eN(P-I)$, and
$\frac{1}{n}S_n=\frac{1}{n}\sum_{i=0}^{n-1}P^i$ is the Ces\`aro mean of $P$. We obtain the following estimate for
$\tau_{A,B}=\mu_{A,B}-I$. For contractions $A$ and $B$ on a Banach space, see \cite[Proposition~2.3(i)]{DDK14}, and
on a Hilbert space, see \cite[Corollary~1.4]{Tur}.

\begin{corollary} \label{cor05}
Let $P\in\pB(\eY)$ be such that $\liminf_{n\to\infty}\frac{1}{n}\|P^n\|=0$ and
$$K=\sup_{n\geq 1}\Bigl\|\frac{1}{n}\sum_{i=0}^{n-1}P^i\Bigr\|<\infty.$$
Then
$$\sin\bigl(\angle(\eN(P-I),\eR(P-I))\bigr)\geq\frac{1}{K}.$$
In particular, if $A,B\in\pB(\eX)$ are power-bounded, then
$$\sin\theta_{\tau_{A,B}}\geq\Bigl(\sup_{n\geq 0}\|A^n\|\,\sup_{n\geq 0}\|B^n\|\Bigr)^{-1}\qquad\text{and}\qquad
\asc(\tau_{A,B})\leq 1.$$
\end{corollary}

\begin{proof}
If $\eN(P-I)=\{0\}$, then $\sin\bigl(\angle(\eN(P-I),\eR(P-I))\bigr)=1\geq\frac{1}{K}$, because $K\geq\|I\|=1$.
Otherwise, choose a strictly increasing sequence $(n_k)_{k=1}^\infty$ with $\frac{1}{n_k}\|P^{n_k}\|\to 0$ and
apply Theorem~\ref{theo03} with $Q=I$ and $M=1$. If $A$ and $B$ are power-bounded, then
$\|\mu_{A,B}^n\|\leq\sup_j\|A^j\|\sup_j\|B^j\|$ for all $n\geq 0$. Therefore, the first part applies to $P=\mu_{A,B}$.
\end{proof}

The next corollary is a Banach space version of the orthogonality theorem for an isometry and a contraction.
For Hilbert space operators, see \cite[Theorem~2.2]{BB}, and for $A=B$ an isometry on a Hilbert space,
see \cite{And}. On a Hilbert space, the corollary also follows from \cite[Proposition~2.3(ii)]{DDK14}, whose
Banach space version assumes instead that $A$ is left invertible by a contraction.

\begin{corollary} \label{cor07}
If $A\in\pB(\eX)$ is an isometry and $B\in\pB(\eX)$ is a contraction, then
$\eN(\delta_{A,B})\perp\eR(\delta_{A,B})$.
\end{corollary}

\begin{proof}
We may assume that $\eN(\delta_{A,B})\ne\{0\}$. Apply Corollary~\ref{cor04} to $P=L_A$ and $Q=R_B$. Both are contractions,
and hence $K_1=K_2=1$. Moreover, $\|L_A^nT\|=\|A^nT\|=\|T\|$ for every $T\in\pB(\eX)$, and thus $M=1$.
Therefore $\sin\theta_{\delta_{A,B}}=1$.
\end{proof}

An operator $A\in\pB(\eX)$ is said to be doubly-bounded if it has a left inverse $A^l$ (that is, $A^lA=I$) such
that both $A$ and $A^l$ are power-bounded.

\begin{corollary} \label{cor02}
Let $A\in\pB(\eX)$ be doubly-bounded with a left inverse $A^l$, and let $B\in\pB(\eX)$ be power-bounded. Then
$$\sin\theta_{\delta_{A,B}}\geq\Bigl(\sup_{n\geq 0}\|A^n\|\,\sup_{n\geq 0}\|(A^l)^n\|\,\sup_{n\geq 0}\|B^n\|\Bigr)^{-1},$$
and consequently $\asc(\delta_{A,B})\leq 1$.
\end{corollary}

\begin{proof}
Let $K_1=\sup_{n\geq 0}\|A^n\|$, $K_2=\sup_{n\geq 0}\|B^n\|$ and $K_3=\sup_{n\geq 0}\|(A^l)^n\|$. We
have $\|A^n\|\|(A^l)^n\|\geq\|(A^l)^nA^n\|=1$ for all $n\geq 0$ and $\|B^0\|=1$. Thus $K_1K_2K_3\geq 1$. Hence the
inequality holds if $\eN(\delta_{A,B})=\{0\}$, and we may assume that $\eN(\delta_{A,B})\ne\{0\}$. Since
$(A^l)^nA^n=I$, we have $\|T\|\leq K_3\|A^nT\|$ for all $T\in\pB(\eX)$ and $n\geq 0$. Thus $L_A$ is
power-bounded below with the constant $K_3^{-1}$ on $\pB(\eX)$, and hence on $\eN(\delta_{A,B})$. Now apply
Corollary~\ref{cor04} to $P=L_A$ and $Q=R_B$.
\end{proof}

\section{Holomorphic functional calculus} \label{Sec06}

For an open set $\Omega\subseteq\bC$, let $\Hol(\Omega)$ be the algebra of all holomorphic functions on $\Omega$.
If $T\in\pB(\eY)$ and $\sigma(T)\subseteq\Omega$, then $f(T)$ is defined for $f\in\Hol(\Omega)$ by the Riesz
functional calculus \cite[Ch.~VII, \S 4]{Con}. In particular, $f(T)$ commutes with every operator commuting
with $T$, $\sigma(f(T))=f(\sigma(T))$, and $(g\circ f)(T)=g(f(T))$ whenever $g$ is holomorphic on a
neighborhood of $f(\sigma(T))$.

The following theorem is an abstract form of Weber's argument \cite[Theorem~1]{Web}. It works for commuting
operators on any Banach space.

\begin{theorem} \label{theo02}
Let $P,Q\in\pB(\eY)$ be commuting operators, let $\Omega$ be an open neighborhood of $\sigma(P)\cup\sigma(Q)$,
and let $f\in\Hol(\Omega)$. Let $\Gamma$ be a cycle in $\Omega\setminus(\sigma(P)\cup\sigma(Q))$ such that
$\ind_\Gamma(z)=1$ for $z\in\sigma(P)\cup\sigma(Q)$ and $\ind_\Gamma(z)=0$ for $z\notin\Omega$, and let
$$\Phi_f=\frac{1}{2\pi i}\int_\Gamma f(z)(z-P)^{-1}(z-Q)^{-1}\,dz.$$
Then $\Phi_f$ commutes with every operator commuting with $P$ and $Q$, and
\begin{equation} \label{eq04}
f(P)-f(Q)=(P-Q)\Phi_f=\Phi_f(P-Q).
\end{equation}
Consequently, for every $k\geq 1$,
\begin{equation} \label{eq05}
\begin{aligned}
&\eR\bigl((f(P)-f(Q))^k\bigr)\subseteq\eR\bigl((P-Q)^k\bigr)\qquad\text{and}\\
&\eN\bigl((P-Q)^k\bigr)\subseteq\eN\bigl((f(P)-f(Q))^k\bigr).
\end{aligned}
\end{equation}
If, in addition, $f$ is injective on $\Omega$, then equality holds in both inclusions of \eqref{eq05} for every
$k\geq 1$. Hence $\asc(f(P)-f(Q))=\asc(P-Q)$, $\dsc(f(P)-f(Q))=\dsc(P-Q)$,
$\angle\bigl(\eN(f(P)-f(Q)),\eR(f(P)-f(Q))\bigr)=\angle\bigl(\eN(P-Q),\eR(P-Q)\bigr)$, and for every $k\geq 1$
the range of $(f(P)-f(Q))^k$ is closed if and only if the range of $(P-Q)^k$ is closed.
\end{theorem}

\begin{proof}
The first assertion is clear, since an operator commuting with $P$ and $Q$ commutes with their resolvents. For
$z\in\Gamma$ we have
\begin{align*}
(z-P)^{-1}-(z-Q)^{-1}&=(z-P)^{-1}\bigl((z-Q)-(z-P)\bigr)(z-Q)^{-1}\\
&=(P-Q)(z-P)^{-1}(z-Q)^{-1},
\end{align*}
where we used that $P-Q$ commutes with $(z-P)^{-1}$. Multiplying by $\frac{f(z)}{2\pi i}$ and integrating over
$\Gamma$ gives \eqref{eq04}. Since $\Phi_f$ commutes with $P-Q$, we get
$(f(P)-f(Q))^k=(P-Q)^k\Phi_f^k=\Phi_f^k(P-Q)^k$, which yields \eqref{eq05}.

Let $f$ be injective on $\Omega$. Then $f(\Omega)$ is open, and the inverse function $g\colon f(\Omega)\to\Omega$
is holomorphic. The operators $f(P)$ and $f(Q)$ commute, $\sigma(f(P))\cup\sigma(f(Q))=f(\sigma(P)\cup\sigma(Q))
\subseteq f(\Omega)$, and $g(f(P))=P$, $g(f(Q))=Q$. Applying \eqref{eq05} to $f(P)$, $f(Q)$ and $g$, we obtain the
reverse inclusions. The equality of the null spaces of all powers gives the equality of the ascents, the
equality of the ranges of all powers gives the equality of the descents and the last assertion, and the
equality of the null spaces and of the ranges gives the equality of the angles.
\end{proof}

Let $A,B\in\pB(\eX)$. The operators $P=L_A$ and $Q=R_B$ on $\pB(\eX)$ commute, $\sigma(L_A)\subseteq\sigma(A)$
and $\sigma(R_B)\subseteq\sigma(B)$, since $(z-L_A)^{-1}=L_{(z-A)^{-1}}$ and $(z-R_B)^{-1}=R_{(z-B)^{-1}}$ for
$z\notin\sigma(A)$ and $z\notin\sigma(B)$, respectively. As the maps $A\mapsto L_A$ and $B\mapsto R_B$ are
bounded and linear, we get $f(L_A)=L_{f(A)}$ and $f(R_B)=R_{f(B)}$ for $f$ holomorphic on a neighborhood of
$\sigma(A)\cup\sigma(B)$. Hence Theorem~\ref{theo02} gives the following extension of Weber's theorem.

\begin{corollary} \label{cor03}
Let $A,B\in\pB(\eX)$, let $\Omega$ be an open neighborhood of $\sigma(A)\cup\sigma(B)$ and let $f\in\Hol(\Omega)$.
Then $\delta_{f(A),f(B)}=\delta_{A,B}\Phi_f=\Phi_f\delta_{A,B}$, where
$$\Phi_f(T)=\frac{1}{2\pi i}\int_\Gamma f(z)(z-A)^{-1}T(z-B)^{-1}\,dz\qquad(T\in\pB(\eX)),$$
and $\Gamma$ is a cycle as in Theorem~\ref{theo02} with $\sigma(A)\cup\sigma(B)$ in place of $\sigma(P)\cup\sigma(Q)$.
In particular,
$$\eR(\delta_{f(A),f(B)})\subseteq\eR(\delta_{A,B})\qquad\text{and}\qquad\eN(\delta_{A,B})\subseteq\eN(\delta_{f(A),f(B)}).$$
If $f$ is injective on $\Omega$, then $\eR(\delta^k_{f(A),f(B)})=\eR(\delta^k_{A,B})$ and
$\eN(\delta^k_{f(A),f(B)})=\eN(\delta^k_{A,B})$ for all $k\geq 1$. Consequently,
$\delta_{f(A),f(B)}$ and $\delta_{A,B}$ have the same ascent and the same descent, and
$\theta_{\delta_{f(A),f(B)}}=\theta_{\delta_{A,B}}$.\qed
\end{corollary}

\subsection*{Acknowledgements}
The first two authors were partially supported by the Slovenian Research and Innovation Agency through the
research programs P2-0268, P1-0285, and research projects N1-0428, J1-70046, J1-70047, and J1-50000.
The third author was supported by the Ministry of Science, Technological Development, and Innovation of
the Republic of Serbia [Grant Number: 451-03-34/2026-03/200102].
\subsection*{Authors' contributions}
The authors declare that they have contributed equally to this paper. All authors have read and approved this
version.
\subsection*{AI Declaration}
During the preparation of this work, the authors used Claude (Anthropic, model Claude Opus 5.5) to identify relevant references, to verify
the mathematical arguments, to assist in the derivation and drafting of parts of the results, and for language
editing and improvement of the exposition. ChatGPT (OpenAI) was additionally consulted for pre-submission checks.
The authors reviewed, verified, and edited all generated content and take full responsibility for the content of
the published article.
%


\begin{thebibliography}{99}

\bibitem{And}
	J. Anderson,
	\textit{On normal derivations},
	Proc. Amer. Math. Soc. \textbf{38} (1973), 135--140.

\bibitem{AF}
	J. Anderson, C. Foia\c{s},
	\textit{Properties which normal operators share with normal derivations and related operators},
	Pacific J. Math. \textbf{61} (1975), no. 2, 313--325.

\bibitem{BA}
	E. Boasso, M. Amouch,
	\textit{Generalized Browder's and Weyl's theorems for generalized derivations},
	Mediterr. J. Math. \textbf{12} (2015), no. 1, 117--131.

\bibitem{BD}
	F. F. Bonsall, J. Duncan,
	\textit{Numerical Ranges II},
	London Math. Soc. Lecture Note Ser. 10, Cambridge University Press, Cambridge, 1973.

\bibitem{BB}
	S. Bouali, Y. Bouhafsi,
	\textit{On the range-kernel orthogonality of elementary operators},
	Math. Bohem. \textbf{140} (2015), no. 3, 261--269.

\bibitem{BK}
	J. Bra\v{c}i\v{c}, B. Kuzma,
	\textit{Localizations of the Kleinecke-Shirokov theorem},
	Oper. Matrices \textbf{1} (2007), no. 3, 385--389.

\bibitem{Con}
	J. B. Conway,
	\textit{A Course in Functional Analysis}, 2nd edition,
	Grad. Texts in Math. 96, Springer-Verlag, New York, 1990.

\bibitem{Du98}
	B. P. Duggal,
	\textit{A remark on normal derivations},
	Proc. Amer. Math. Soc. \textbf{126} (1998), no. 7, 2047--2052.

\bibitem{Du04}
	B. P. Duggal,
	\textit{Subspace gaps and range-kernel orthogonality of an elementary operator},
	Linear Algebra Appl. \textbf{383} (2004), 93--106.

\bibitem{Du05}
	B. P. Duggal,
	\textit{Subspace gaps and Weyl's theorem for an elementary operator},
	Int. J. Math. Math. Sci. \textbf{2005}, no. 3, 465--474.

\bibitem{DDK14}
	B. P. Duggal, S. V. Djordjevi\'{c}, C. S. Kubrusly,
	\textit{Elementary operators, finite ascent, range closure and compactness},
	Linear Algebra Appl. \textbf{449} (2014), 334--349.

\bibitem{DDK14b}
	B. P. Duggal, S. V. Djordjevi\'{c}, C. S. Kubrusly,
	\textit{Algebraic elementary operators},
	Funct. Anal. Approx. Comput. \textbf{6} (2014), no. 2, 43--50.

\bibitem{Jac}
	N. Jacobson,
	\textit{Rational methods in the theory of Lie algebras},
	Ann. of Math. \textbf{36} (1935), 875--881.

\bibitem{Jam}
	R. C. James,
	\textit{Orthogonality and linear functionals in normed linear spaces},
	Trans. Amer. Math. Soc. \textbf{61} (1947), 265--292.

\bibitem{Kec}
	D. Ke\v{c}ki\'{c},
	\textit{Orthogonality of the range and the kernel of some elementary operators},
	Proc. Amer. Math. Soc. \textbf{128} (2000), no. 11, 3369--3377.

\bibitem{Kit}
	F. Kittaneh,
	\textit{Normal derivations in norm ideals},
	Proc. Amer. Math. Soc. \textbf{123} (1995), no. 6, 1779--1785.

\bibitem{Kle}
	D. C. Kleinecke,
	\textit{On operator commutators},
	Proc. Amer. Math. Soc. \textbf{8} (1957), 535--536.

\bibitem{Kob}
	H. Kober,
	\textit{A theorem on Banach spaces},
	Compositio Math. \textbf{7} (1940), 135--140.

\bibitem{Kre}
	U. Krengel,
	\textit{Ergodic Theorems},
	de Gruyter Stud. Math. 6, Walter de Gruyter, Berlin, 1985.

\bibitem{LM}
	F. Lombarkia, S. Megri,
	\textit{Generalized derivation, SVEP, finite ascent, range closure},
	Filomat \textbf{34} (2020), no. 10, 3473--3482.

\bibitem{Mul}
	V. M\"{u}ller,
	\textit{Spectral Theory of Linear Operators and Spectral Systems in Banach Algebras}, 2nd edition,
	Oper. Theory Adv. Appl. 139, Birkh\"{a}user, Basel, 2007.

\bibitem{Ros}
	M. Rosenblum,
	\textit{On the operator equation $BX-XA=Q$},
	Duke Math. J. \textbf{23} (1956), 263--269.

\bibitem{Rot}
	W. E. Roth,
	\textit{On direct product matrices},
	Bull. Amer. Math. Soc. \textbf{40} (1934), 461--468.

\bibitem{Shi}
	F. V. Shirokov,
	\textit{Proof of a conjecture by Kaplansky},
	Uspekhi Mat. Nauk \textbf{11} (1956), no. 4, 167--168 (in Russian).

\bibitem{Sin}
	A. M. Sinclair,
	\textit{Eigenvalues in the boundary of the numerical range},
	Pacific J. Math. \textbf{35} (1970), no. 1, 231--234.

\bibitem{Sta}
	H. Stankovi\'{c},
	\textit{Kleinecke-Shirokov theorem: a version for isometric transformations},
	Anal. Math. Phys. \textbf{15} (2025), no. 3, Paper No. 54, 9 pp.

\bibitem{TL}
	A. E. Taylor, D. C. Lay,
	\textit{Introduction to Functional Analysis}, 2nd edition,
	John Wiley \& Sons, New York, 1980.

\bibitem{Tur}
	A. Turn\v{s}ek,
	\textit{Elementary operators and orthogonality},
	Linear Algebra Appl. \textbf{317} (2000), 207--216.

\bibitem{Web}
	R. E. Weber,
	\textit{Analytic functions, ideals, and derivation ranges},
	Proc. Amer. Math. Soc. \textbf{40} (1973), 492--496.

\end{thebibliography}
\end{document}